\documentclass[a4paper,11pt]{scrartcl}
\usepackage[utf8]{inputenc}
\usepackage{lineno}
\usepackage{hyperref}
\usepackage{tikz}
\usepackage{hyphenat}
\usetikzlibrary{calc,arrows}
\usetikzlibrary{patterns}
\usepackage{amsmath,amssymb,amsthm}
\usepackage{enumerate}
\usepackage{url}

\newtheorem{thm}{Theorem}[section]
\newtheorem{dfn}[thm]{Definition}
\newtheorem{lem}[thm]{Lemma}
\newtheorem{prop}[thm]{Proposition}

\theoremstyle{remark}
\newtheorem{rem}[thm]{Remark}
\newtheorem{ex}[thm]{Example}

\DeclareMathOperator{\scal}{Scal}

\DeclareMathOperator{\diam}{diam}

\DeclareMathOperator{\aire}{area}

\DeclareMathOperator{\sys}{sys}
\DeclareMathOperator{\str}{stretch}

\newcommand{\eps}{\varepsilon}
\newcommand{\ph}{\varphi}

\title{On the 2-systole of stretched enough positive scalar curvature metrics on $\mathbb{S}^2\times\mathbb{S}^2$}
\author{Thomas Richard\footnote{LAMA, Univ Paris Est Creteil, Univ Gustave Eiffel, CNRS, F-94010, Créteil, France}}
\begin{document}
\maketitle

\begin{abstract}
  We use recent developments by
  Gromov and Zhu to derive
  an upper bound for the 2-systole of the homology class of
  $\mathbb{S}^2\times\{\ast\}$ in a $\mathbb{S}^2\times\mathbb{S}^2$
  with a positive scalar curvature metric such that the set of
  surfaces homologous to $\mathbb{S}^2\times\{\ast\}$ is wide enough in
  some sense.
\end{abstract}

Recall that the systole of a compact Riemannian manifold $(M^n,g)$ is the
length of the shortest non contractible loop in $(M^n,g)$. In the
middle of the 20th century, Loewner and Pu proved sharp upper bounds
on the systole of any metric on $\mathbb{T}^2$ or $\mathbb{RP}^2$ in
term of its volume.
These were vastly generalized when in the early 80s Gromov gave similar
(non sharp) bounds for $n$-dimensional essential manifolds. (See
\cite{berger1993systoles} for the full story until 1993.)

The $k$-systole $\sys_k(g)$ of $(M^n,g)$ is the infimum  of the
$k$-dimensional volume over all homologically
non trivial $k$-cycles. In general, for
$k\geq 2$, the $k$-systole of a manifold cannot be bounded by the
volume alone. In particular, Katz and Suciu showed in
\cite{katz1999volume} that one can find
metrics on $\mathbb{S}^2\times\mathbb{S}^2$ whose volume is $1$ but
whose 2-systole can be arbitrary large. (Similar examples in higher
dimension where already known to Gromov, see again \cite{berger1993systoles}.)

One way to circumvent this is to introduce the more subtle ``stable
systoles'' (see again \cite{berger1993systoles}), another way is to
try to introduce curvature restrictions. This second route was first
considered in dimension $3$ by Bray, Brendle and Neves for manifolds
such as $\mathbb{S}^2\times\mathbb{S}^1$ with positive scalar
curvature metrics in
\cite{bray2010rigidity}. Recently, Zhu treated the case of
$\mathbb{S}^2\times\mathbb{T}^n$ ($n+2\leq 7$) in
\cite{zhu2020rigidity}, see Section \ref{sec:results-gromov-zhu} for a
precise statement. 

Here we will show how Zhu's result together with recent developments
due to Gromov gives some progress in the case
of $\mathbb{S}^2\times\mathbb{S}^2$.

Let $\mathcal{S}_\ell$ be the set of embedded surfaces in
$\mathbb{S}^2\times\mathbb{S}^2$ which are in the same homology class
as $\mathbb{S}^2\times\{\ast\}$.

\begin{dfn}
  Let $g$ be a Riemannian metric on
  $\mathbb{S}^2\times\mathbb{S}^2$, the left stretch of $g$, denoted
  by $\str_{\ell}(g)$ is defined as
  \[\str_{\ell}(g)=\sup_{S_1,S_2\in\mathcal{S}_\ell}d_g\left(S_1,S_2\right).\]
\end{dfn}
\begin{ex}
  For a product metric $g=g_1\oplus g_2$ on
  $\mathbb{S}^2\times\mathbb{S}^2$, the left stretch is
  $\str_{\ell}(g)=\diam(\mathbb{S}^2,g_2)$ and is achieved by
  $S_1=\mathbb{S}^2\times\{p_1\}$ and $S_2=\mathbb{S}^2\times\{p_2\}$
  where $d_{g_2}(p_1,p_2)=\diam(\mathbb{S}^2,g_2)$.
\end{ex}
As the example above shows, the left stretch is a measure of the
1-dimensional size of the right factor of $(\mathbb{S}^2\times
\mathbb{S}^2,g)$. It is bounded above by the diameter of $g$.
\begin{dfn}\label{dfn-str}
  Let $g$ be a Riemannian metric on
  $\mathbb{S}^2\times\mathbb{S}^2$, the left 2-systole of $g$, denoted
  by $\sys_{2,\ell}(g)$ is defined as
  \[\sys_{2,\ell}(g)=\inf_{S\in\mathcal{S}_\ell}\aire_g(S).\]
\end{dfn}
\begin{ex}
  For a product metric $g=g_1\oplus g_2$ on
  $\mathbb{S}^2\times\mathbb{S}^2$, the left 2-systole is 
  $\sys_{2,\ell}(g)=\aire(\mathbb{S}^2,g_1)$ and is achieved by
  any $S=\mathbb{S}^2\times\{\ast\}$.
\end{ex}

The theorem below is the main result of the paper. It says that
positive scalar curvature metrics on $\mathbb{S}^2\times\mathbb{S}^2$
with large left stretch cannot have large left 2-systole.
\begin{thm}\label{thm-main}
  Let $g$ be a metric on $\mathbb{S}^2\times\mathbb{S}^2$ with
  $\scal_g\geq 4$.

  If $s=\str_{\ell}(\mathbb{S}^2\times \mathbb{S}^2,g) >
  \tfrac{\sqrt{3}\pi}{2}$, then $\sys_{\ell}(\mathbb{S}^2\times
  \mathbb{S}^2,g)\leq \frac{8\pi s^2}{4s^2-3\pi^2}$. Moreover there is
  an embedded $2$-sphere whose area is at most $\frac{8\pi s^2}{4s^2-3\pi^2}$.
\end{thm}
\begin{rem}
  This estimate is asymptotically sharp when $s$ goes to $+\infty$, as
  the example of the product fo two round spheres of radi
  $\frac{1}{\sqrt{k}}$ and $\frac{1}{\sqrt{2-k}}$ for $k\in[1,2)$
  shows. It remains unclear to the author wether this estimate is optimal or not. What can be shown from
  tracking down the equality case in the proofs of Theorems
  \ref{thm-zhu} and \ref{thm-gro} is that if the inequality in Theorem
  \ref{thm-main} is an equality then M contains an hypersurface
  $\Sigma$ such that:
  \begin{itemize}
  \item There exists a map $\Sigma\to\mathbb{S}^2\times\mathbb{S}^1$
    with non-vanishing degre.
  \item The universal cover of $\Sigma$ is isometric to $\mathbb{S}^2\times\mathbb{R}$ with the product metric such that
    the $\mathbb{S}^2$ factor is round and has area exactly $\frac{8\pi s^2}{4s^2-3\pi^2}$.
  \end{itemize}
  Products of two round $\mathbb{S}^2$ which satisfy the properties
  above have left stretch smaller than $s$. The investigation of
  doubly warped product metrics of the form $\psi^2(r)g_{\mathbb{S}^2}+\ph^2(r)d\theta^2+ dr^2$
  on $\mathbb{S}^2\times \mathbb{S}^2$ was inconclusive too. 
\end{rem}
\begin{rem}
  Of course, if one denotes by $\mathcal{S}_r$ the set of embedded
  spheres homologous to $\{\ast\}\times\mathbb{S}^2$, we can define
  right stretch and systole as
  $\str_{r}(g)=\sup_{S_1,S_2\in\mathcal{S}_r}d_g\left(S_1,S_2\right)$
  and $\sys_{2,r}(g)=\inf_{S\in\mathcal{S}_r}\aire_g(S)$ and get the
  same inequality between the right systole and the right stretch.
\end{rem}

It is currently unknown wether the $2$-systole is
bounded from above on the set
$\mathcal{M}_{\mathbb{S}^2\times\mathbb{S}^2,\scal\geq 4}$ of metrics on
$\mathbb{S}^2\times\mathbb{S}^2$ with scalar curvature greater than
$4$. Our result says that if there is no upper bound, there
must be a sequence of metrics in $\mathcal{M}_{\mathbb{S}^2\times\mathbb{S}^2,\scal\geq 4}$
whose $2$-systoles goes to infinity while their left and right
stretches stay below $\tfrac{\sqrt{3}}{2}\pi$.

The rest of the paper is organized as follows: we first review some
previously known results by Gromov and Zhu on manifolds with positive scalar curvature
which will be used to prove Theorem \ref{thm-main}, we then prove an
elementary topological fact and we finally give the proof of the main theorem.

\subsubsection*{Acknowledgements}

The author would like to thank G. Besson, S. Maillot and S. Sabourau for helpful
discussions. He also thanks Jintian Zhu and the anonymous referees for pointing an inaccuracies in a previous
version of this work. The author
is supported by the grant ANR-17-CE40-0034 of the French National
Research Agency ANR (Project CCEM). 

\section{Positive scalar curvature inequalities by Gromov and Zhu}
\label{sec:results-gromov-zhu}

We will need the following result by Jintian Zhu, which was already
alluded to in the introduction.

\begin{thm}[\cite{zhu2020rigidity}, Theorem 1.2]\label{thm-zhu}
  Let $(M^n,g)$ be a closed manifold of dimension at most seven such
  that:
  \begin{itemize}
  \item $\scal_g\geq 2$;
  \item there exists a map $F:M^n\to\mathbb{S}^2\times\mathbb{T}^{n-2}$
    with non-vanishing degree.
  \end{itemize}
  Then $\sys_2(g)\leq 4\pi$. More precisely one can find an
  embedded 2-sphere $S$ such that
  $F_*([S])=[\mathbb{S}^2\times\{\ast\}]$ and whose area is less than $4\pi$.
\end{thm}
The proof uses the second variation formula of the area functional for stable minimal hypersurfaces and a
repeated symmetrization construction to reduce the problem to
$\mathbb{T}^{n-2}$-invariant metrics on
$\Sigma^2\times\mathbb{T}^{n-2}$. This method of proof was pioneered
by Fischer-Colbrie and Schoen in dimension 3 and used in 
dimensions at most seven by Gromov and Lawson in
\cite{gromov1983positive}. Zhu also studies the equality case: under
the hypothesis of the theorem, if there is an embedded $2$-sphere $S$
such that $F_* ([S]) = [\mathbb{S}^2 × \{*\}]$ and whose area is $4\pi$, then the universal cover of $M$ is a product of the round $2$-sphere of area $4\pi$ and a euclidean space.

\begin{rem}\label{rem-zhu}
  The proof in \cite{zhu2020rigidity} shows that if one starts with an
  $\mathbb{S}^1$-invariant metric on $M=M'\times\mathbb{S}^1$
  satisfying the hypothesis of Theorem \ref{thm-zhu}, then one can
  find a $2$-sphere in $M'$ whose area is at most $4\pi$.
\end{rem}

In the Spring of 2019, the author was lucky enough to attend the
lectures ``Old, New and Unknown around Scalar Curvature'' given by
Gromov at IHES. There, the author learned about the following result by
Gromov which is also proved using the Fischer-Colbrie--Schoen
symmetrization process:
\begin{thm}[\cite{gromov2018metric}, p. 2]
  Let $2\leq n\leq 7$, $M=[-1,1]\times \mathbb{T}^{n-1}$ and
  $\partial_{\pm}M=\{\pm 1\}\times \mathbb{T}^{n-1}$. 
  Let $g$ be a Riemannian metric on $M$
  such that $\scal_g\geq n(n-1)$. Then
  \[d_g\big(\partial_-M,\partial_+M\big)\leq\frac{2\pi}{n}. \]
\end{thm}
The equality case is also studied. If $d_g\big(\partial_-M,\partial_+M\big)=\frac{2\pi}{n}$, then $M$ is isometric to
$[-\tfrac{\pi}{n},\tfrac{\pi}{n}]\times\mathbb{T}^{n-1}$
with the product metric $dt^2 + \cos^{\tfrac{4}{n}}(\frac{nt}{2})dx^2$
where $dx^2$ is a flat metric on $\mathbb{T}^{n-1}$.

In his Spring 2019 lectures, Gromov gave a new proof of the previous
theorem. Instead of using the second variation of the
$(n-1)$-dimensional volume as in \cite{gromov2018metric}, a twisted
functional is considered. Given a density $\mu:M^n\to\mathbb{R}$, the
$\mu$-area functional maps an open set $U\subset M$ to
$\mathcal{V}_{n-1}(\partial U)-\int_U\mu$, where $\mathcal{V}_{n-1}$
denotes the $(n-1)$-dimensional volume. Using a well chosen
density $\mu$, Gromov proved the following theorem, which is a
generalisation of the theorem above.
\begin{thm}[\cite{gromov2019four}, Section 5.3]\label{thm-gro}
  Let $(M^n,g)$ $(n\leq 7)$ be a compact $n$-manifold with two boundary components $\partial_-M$
  and $\partial_+M$. Assume that
  \[\scal_g\geq \frac{4(n-1)\pi^2}{n\, d_g(\partial_-M,
      \partial_+M)^2}+\delta\]
  for some $\delta>0$. Then there exists
  \begin{itemize}
  \item an hypersurface $\Sigma$ which separates
    $\partial_-M$ and $\partial_+M$,
  \item a positive function
    $u:\Sigma\to\mathbb{R}$,
  \end{itemize}
  such that the metric
  $h=g|_{\Sigma}+u^2dt^2$ on $\Sigma\times\mathbb{R}$ has $\scal_h\geq \delta$.
\end{thm}

\section{Topological preliminaries}
\label{sec:topol-prel}

Before proving the main theorem, we establish a topological
preliminary. Let $M=\mathbb{S}^2\times\mathbb{S}^2$.

The idea is that if we remove two disjoint spheres $S_1$ and $S_2$ in
$\mathcal{S}_\ell$ from $M$, we should be in a situation which is very
similar to removing $\mathbb{S}^2\times\{n\}$ and
$\mathbb{S}^2\times\{s\}$ from $\mathbb{S}^2\times\mathbb{S}^2$, where
$n$ and $s$ are the north and south
poles of the right factor. Thus $\tilde M=M\backslash (S_1\cup S_2)$
should look like
\[\mathbb{S}^2\times\mathbb{S}^2\backslash\left(\mathbb{S}^2\times\{n\}\cup
    \mathbb{S}^2\times\{s\}\right)=\mathbb{S}^2\times
  \left(\mathbb{S}^2\backslash\{n,s\}\right)\simeq \mathbb{S}^2\times
  \mathbb{S}^1\times (-1,1).\]

However, our surfaces $S_1$ and $S_2$ from $\mathcal{S}_\ell$ may have
higher genus and thus $\tilde M$ may be more complicated. 

However we are able to prove the following, which will be enough for our purpose:

\begin{prop}\label{prop-top}
  Let $S_1$ and $S_2$ be two disjoint surfaces in
  $\mathcal{S}_\ell$. Let $\Sigma\subset M$ be a connected hypersurface such
  that
  \begin{itemize}
  \item $\Sigma$ is disjoint from $S_1$ and $S_2$,
  \item $\Sigma$ separates $S_1$ and $S_2$.
  \end{itemize}
  Then there exists a non-zero degree map $\Sigma\to\mathbb{S}^2\times\mathbb{S}^1$.
\end{prop}
This map will be the restriction of a map $F:\tilde M=M\backslash
(S_1\cup S_2)\to \mathbb{S}^2\times\mathbb{S}^1$.

We will first show :
\begin{lem}\label{lem-norm}
  Any $S\in \mathcal{S}_\ell$ has a trivial normal bundle.
\end{lem}
\begin{proof}
  Let $NS$ be the normal bundle to $S$ and let $e(NS)$ be the Euler
  class of $NS$. It follows from \cite{BottTu}[Theorem 11.17,
  p. 125] that the Euler class of a vector bundle over a compact manifold can be computed as
  the intersection number between the zero section and another
  transverse section times the fundamental class of the base. 

  Since $NS$ can be embedded in $M$ as a small tubular neighborhood of
  $N$, this intersection number can be computed in $M$. Since in
  $\mathcal{S}_\ell$ one can find two disjoint spheres, the
  intersection  number is 0 and $e(NS)=0$.

  Moreover the Euler class is the only obstruction for an oriented
  rank $k$ vector bundle over a a compact $k$-manifold to have a non
  vanishing section (see \cite{hatcher2003vector}, Proposition
  3.22). Hence $NS$ has a nowhere vanishing section. Since $NS$ is an
  orientable rank 2 vector bundle, it is trivial.
\end{proof}

For $\eps>0$, we denote by $S_i^\eps$ the tubular neighborhood around
$S_i$. We will choose $\eps$ small enough so that $S_1^\eps$ and
$S_2^\eps$ are disjoint regular neighborhoods.

\begin{lem}\label{lem-S2}
  The map $F_2:\tilde M\subset \mathbb{S}^2\times\mathbb{S}^2\to\mathbb{S}^2\times\{\ast\}$ is
  surjective in homology.
\end{lem}
\begin{proof}
  Consider a $3$-cycle $C$ in $M$ such that $\partial C= S_1-S_2$. By using
  an fine enough triangulation, we can decompose $C=C_1+\tilde C+C_2$,
  where $\tilde C$ is a $3$-cycle in $\tilde M$ and $C_i$ is a
  $3$-cycle in $S_i^\eps$. Now, $S'_1=\partial C_1-S_1$ is a $2$-cycle in
  $\tilde M$ homologous to $S_1$, hence $\pi(F_2)_*([S'_1])=[\mathbb{S}^2\times\{\ast\}]$.  
\end{proof}

Set $S=S_1\cup S_2$, $S^\eps=S_1^\eps\cup S_2^\eps$ and $\tilde
S^\eps=S^\eps\backslash S$.

\begin{lem}\label{lem-H1}
  $H^1(\tilde M,\mathbb{Z})=\mathbb{Z}$.
\end{lem}
\begin{proof}
  We will use the Mayer-Vietoris exact sequence for cohomology
  associated with the decomposition $M=S^\eps\cup \tilde M$. Note that
  $S^\eps$ is homotopy equivalent to $S$ and
  since $S_1$ and $S_2$ have trivial normal bundle, $\tilde
  S^\eps$ is homotopy equivalent to $S\times\mathbb{S}^1$.


  The Mayer Vietoris sequence gives 
  \[\cdots\to H^1(M)\to H^1(S^\eps)\oplus
    H^1(\tilde M) \to H^1(\tilde S^\eps)\to H^2(M)\to\cdots\]
  Thus, we have 
  \[\cdots\to 0\to H^1(S_1)\oplus H^1(S_2)\oplus
    H^1(\tilde M) \to H^1(S_1)\oplus\mathbb{Z}\oplus
    H^1(S_2)\oplus\mathbb{Z}\to \mathbb{Z}\oplus\mathbb{Z}\to\cdots\] 
  where the map $ H^1(\tilde S^\eps)\simeq H^1(S_1)\oplus\mathbb{Z}\oplus
    H^1(S_2)\oplus\mathbb{Z}\to H^2(M)$ is given by
  $(a,x,b,y)\mapsto (x-y,0)$. Hence $H^1(\tilde M)$ is isomorphic to $\mathbb{Z}$.
\end{proof}

Since $H^1(\tilde M,\mathbb{Z})\simeq \mathbb{Z}$, $H^1(\tilde
M,\mathbb{R})\simeq \mathbb{R}$ and we can find a closed $1$-form
$\alpha$ on $\tilde M$ which is not exact. Let
$\Gamma$ be the image in $\mathbb{R}$ of the abelian group morphism :
$c\in H_1(\tilde M,\mathbb{Z})\mapsto \int_c\alpha\in\mathbb{R}$. By a
computation similar to the proof of the previous lemma, $H_1(\tilde
M,\mathbb{Z})=\mathbb{Z}$. Hence $\Gamma$ is a
discrete subgroup of $\mathbb{R}$. The
classical period map construction gives :
\begin{lem}\label{lem-S1}
  Let $x_0\in\tilde M$. For any curve $c$ from $x_0$ to $x$,
  $\int_c\alpha$ is well defined as an element $\mathbb{R}/\Gamma$. This
  defines a map $F_1:\tilde M\to\mathbb{R}/\Gamma$ such that the induced map
   $H_1(\tilde M)\to H_1(\mathbb{R}/\Gamma) $ is an isomorphism.
\end{lem}

$F=(F_2,F_1):\tilde M\to \mathbb{S}^2\times\mathbb{S}^1$ is the map we
are looking for.

\begin{proof}[Proof of Proposition \ref{prop-top}]
  
First note that any connected hypersurface $\Sigma\subset M$ which separates $S_1$ and $S_2$
will be homologous to $\partial S_1^\eps$ for $\eps$ small
enough. Thus it is enough to show that the image of $[\partial
S_1^\eps]$ under $F$ is not zero in
$H_3(\mathbb{S}^2\times\mathbb{S}^1)=H_2(\mathbb{S}^2)\oplus
H_1(\mathbb{S}^1)$, which is routinely deduced from the proofs of
Lemmas \ref{lem-S2} and \ref{lem-H1}.
\end{proof}

We will also need an elementary fact about the connectedness of separating hypersurfaces. Before stating it let us recall that $\Sigma\subset M$ separates $K_1, K_2 \subset M$ if $K_1$ and $K_2$ lie in different connected components of $M\backslash\Sigma$. The result we will need is the following:

\begin{prop}\label{prop-connect}
  Let $M^n$ be a closed orientable manifold with $H_{n-1}(M) = 0$ and
  let $\Sigma^{n-1}\subset M^n$ be a closed orientable
  hypersurface. Let $K_1$ and $K_2$ be two closed subsets of $M$ such
  that $\Sigma$ separates $K_1$ and $K_2$. Then there is a connected
  component of $\Sigma$ which separates $K_1$ and $K_2$. 
\end{prop}

This is probably well known but the author has not been able to locate a proof in the litterature.

\begin{proof}
  Let us first remark that since $H_{n-1}(M) = 0$, any connected
  closed hypersurface $\Sigma_0\subset M$ satisfies that
  $M\backslash\Sigma_0$ has exactly two connected components whose
  boundary in $M$ is $\Sigma_0$. The fact that $M\backslash\Sigma_0$ has at most two
  connected components follows from the fact $\Sigma_0$ is two sided
  (since both $\Sigma_0$ and $M$ are orientable). If
  $M\backslash\Sigma_0$ was connected, there would be a closed loop
  $\gamma$ which intersects $\Sigma_0$ exactly once, hence the
  intersection number of $[\gamma]\in H_1(M)$ and $[\Sigma_0]\in
  H_{n-1}(M)$ would be non zero, which is absurd since $H_{n-1}(M)=0$.

  We will now prove the following result: let $\tilde\Sigma$ and $\Sigma_0$ be two disjoint hypersurfaces such
  that:
\begin{itemize}
\item $\Sigma_0$ is connected.
\item $\tilde\Sigma\cup\Sigma_0$ separates $K_1$ and $K_2$.
\item $\Sigma_0$ does not separate $K_1$ and $K_2$.
\end{itemize}
Then $\tilde\Sigma$ separates $K_1$ and $K_2$.

Once we have proved this, the proposition follows since we can remove
all the non-separating components form $\Sigma$ until every connected
component of $\Sigma$ separates $K_1$ and $K_2$.

In order to prove the result above, let $U_0$ be the connected
component of $M\backslash\Sigma_0$ which contains $K_1$. Since
$\Sigma_0$ does not separate $K_1$ and $K_2$, $U_0$ also contains
$K_2$. Now let $\tilde U$ be the connected component of
$M\backslash\tilde\Sigma$ which contains $K_1$.

Since the boundaries of $U_0$ and $\tilde U$ are disjoint and both
$U_0$ and $\tilde U$ contain $K_1$ either
$U_0 \subset \tilde U$ or $\tilde U \subset U_0$. In the first case,
$U_0$ will be a connected component of $M\backslash\Sigma$ which 
contains $K_1$ and $K_2$, which is impossible since $\Sigma$ separates
$K_1$ and $K_2$. 

Hence $\tilde U \subset U_0$, which implies that $\tilde U$ is one of
the connected components of $M\backslash\Sigma$. Since $\Sigma$ separates
$K_1$ and $K_2$, $K_2$ cannot be included in $\tilde U$. Hence $\tilde\Sigma$ separates
$K_1$ and $K_2$.
\end{proof}
\section{Proof of the main theorem}
\label{sec:proof}

We are now ready to prove Theorem \ref{thm-main}.
\begin{proof}[Proof of Theorem \ref{thm-main}]
  Our assumption is that $s=\str_{\ell}(\mathbb{S}^2\times \mathbb{S}^2,g) >
  \tfrac{\sqrt{3}\pi}{2}$. Thus, for any $\eps>0$ we can find two surfaces
  $S_1$ and $S_2$ in $\mathcal{S}_\ell$ such that $d_g(S_1,S_2)=
  s-\eps$. Assume that $\eps>0$ is such that
  \begin{itemize}
  \item $s-3\eps>\frac{\sqrt{3}\pi}{2}$;
  \item the tubular neighborhoods $S_1^\eps$ and $S_2^\eps$ of radius
    $\eps$ around $S_1$ and $S_2$ are disjoint and regular.
  \end{itemize}
  Then if we set $\tilde M=M\backslash (S_1^\eps\cup S_2^\eps)$, its
  two boundary components $\partial_-\tilde M=\partial S_1^\eps$ and
  $\partial_+\tilde M=\partial S_2^\eps$ satisfy $d_g(\partial_-\tilde
  M, \partial_+\tilde M)= s-3\eps$.

  Set $\delta=4-\frac{3\pi^2}{(s-3\eps)^2}$ and note that $\delta>0$. Then we have:
  \[\scal_g\geq 4=\frac{3\pi^2}{d_g(\partial_-\tilde
  M, \partial_+\tilde M)^2}+\delta.\]

  This is exactly what we need to apply Gromov's Theorem
  \ref{thm-gro}. Hence we get an hypersurface $\Sigma\subset\tilde M$
  which separates $S_1$ and $S_2$ such that
  $\Sigma\times\mathbb{S}^1$ admits a metric of the form
  $h=g_\Sigma+u^2dt^2$ for some $u:M\to\mathbb{R}$ such that
  $\scal_h\geq\delta$. The hypersurface coming the minimization
  process in the proof of Theorem~\ref{thm-gro} may be disconnected,
  however by Proposition \ref{prop-connect} we can replace $\Sigma$ by
  one of its connected components which separates $S_1$ and $S_2$. 

  Moreover, by Proposition \ref{prop-top}, there exists a non-zero
  degree map $\Sigma\to\mathbb{S}^2\times\mathbb{S}^1$. Hence there is
  a non-zero degree map 
  $\Sigma\times\mathbb{S}^1\to\mathbb{S}^2\times\mathbb{T}^2$ and we
  can apply Zhu's Theorem \ref{thm-zhu} and Remark \ref{rem-zhu} to
  show that one can find an embedded $2$-sphere $S$ in $\Sigma$ whose
  area is at most $\frac{8\pi}{\delta}=\frac{8\pi
    (s-3\eps)^2}{4(s-3\eps)^2-3\pi^2}$. 

  Since $\Sigma$ is isometrically embedded in
  $M$, $S$ embeds in $M$ with area at most $\frac{8\pi
    (s-3\eps)^2}{4(s-3\eps)^2-3\pi^2}$. Moreover, by the properties of the map $F$
  built in section \ref{sec:topol-prel}, the surface $S$ of $M$
  belongs to $\mathcal{S}_\ell$. Since $\eps>0$ can be as small as one
  wants we get the results. 
\end{proof}

\bibliography{2-systoles}{}
\bibliographystyle{alpha}

\end{document}